\documentclass[12pt]{amsart} 
 \usepackage{amssymb, latexsym, amsmath} 

{\normalsize }

\begin{document} 

 \theoremstyle{plain} 
 \newtheorem{theorem}{Theorem}[section] 
 \newtheorem{lemma}[theorem]{Lemma} 
 \newtheorem{corollary}[theorem]{Corollary} 
 \newtheorem{proposition}[theorem]{Proposition} 

\theoremstyle{definition} 
\newtheorem{definition}{Definition}
\newtheorem{example}[theorem]{Example}
\newtheorem{remark}[theorem]{Remark}

\title[Bi-skew braces]{Bi-skew braces and  Hopf Galois structures}

\author{Lindsay N. Childs}
\address{Department of Mathematics and Statistics\\
University at Albany\\
Albany, NY 12222}
\email{lchilds@albany.edu}

\date{\today}

 \begin{abstract} We define a bi-skew brace to be a set $G$ with two group operations $\star$ and $\circ$ so that $(G, \circ, \star)$ is a skew brace  with additive group $(G, \star)$ and also with additive group $(G, \circ)$.  If $G$ is a skew brace, then $G$ corresponds to a Hopf Galois structure of type $(G, \star)$ on any Galois extension of fields with Galois group isomorphic to $(G, \circ)$.  If $G$ is a bi-skew brace, then $G$ also corresponds to a Hopf Galois structure of type $(G, \circ)$ on a Galois extension of fields with Galois group isomorphic to $(G, \star)$.   Many non-trivial examples exist. One source is  radical rings $A$ with $A^3 = 0$,  where one of the groups is abelian and the other need not be.  We find that the left braces of degree $p^3$ classified by Bachiller are bi-skew braces if and only they are radical rings.   A different source of bi-skew braces is semidirect products of arbitrary finite groups, which yield many examples where both groups are non-abelian, and a skew brace proof of a result of Crespo, Rio and Vela that if $G = H\rtimes J$ is a semidirect product of finite groups, then a Galois extension of fields with Galois group $G$ has a Hopf Galois structure of type $H \times J$.
\end{abstract}
\maketitle

\

\

Table of contents:

1.  Introduction, history

2.  Connection with Hopf Galois structures

3.  Bi-skew braces

4.  Bi-skew braces from radical algebras

5.  Bachiller's classification of left braces of order $p^3$

6.  On Bachiller's exponent result

7.  Bi-skew braces and semidirect products

\section{Introduction, history}
\subsection*{Hopf Galois structures}
Let $L/K$ be a $G$-Galois extension of fields:  that is, a Galois extension of fields with Galois group $G= (G, \circ)$.  By E. Artin's classic treatment of Galois theory we know that  $L \otimes L \cong Hom_L(LG, L) = (LG)^*.$  

Chase and Sweedler [CS69] generalized the concept of Galois extension by replacing the group ring of the Galois group by a Hopf algebra, as follows:  suppose $L$ is an $H$-module algebra for $H$ a cocommutative $K$-Hopf algebra $H$.  Then  $L/K$ is an $H$-Hopf Galois extension $L/K$ if
\[L \otimes_K L \cong \mathrm{Hom}_L( L \otimes_K H, L) = (L \otimes H)^*.\]   
Then, as Greither and Pareigis [GP87] observed,  $L\otimes_K L$ 
is a ring direct sum of copies of $L$, so $L \otimes H = LN$ for some subgroup $N$ of $\mathrm{Perm}(G)$ 
 that is normalized by the image $\lambda_\circ(G)$ of the left regular representation $\lambda_\circ:  G \to \mathrm{Perm}(G)$, 
and  $H$ can be recovered by Galois descent:  $H = LN^G$.  

If $L/K$ is a Galois extension with Galois group $G$ and is an $H$-Hopf Galois extension where $L \otimes K H =KN$, then we say that the type of $H$ is the isomorphism type of the group $N$.  

Since $N$ is a regular subgroup of $\mathrm{Perm}(G)$, the map $\varphi: N \to G$ given by $n \mapsto n(e)$ is a bijection, where $e$ is the identity element of $G$.  Let $g = \varphi(n)$.  For $x$ in $G$, define $\lambda_\star:  G \to \mathrm{Perm}(G)$ by \[\lambda_\star(g)(x) = g \star x = \varphi(n)\star x = n(x).\] Then $\star$ is a new group operation on $G$  so that $\varphi:  N \to \lambda_\star (G)$  is an isomorphism, and $\lambda_\star$ is the left regular representation map for this new operation.
 
Thus:
\begin{theorem}\label{111}  Every Hopf Galois structure of type $N$ on a $G = (G, \circ)$-Galois extension yields a second operation $\star$ on $G$ so that $(G, \star) \cong N$.  Then  $\lambda_\circ(G)$ is contained in $\mathrm{Hol}(G, \star)$, the normalizer in $\mathrm{Perm}(G)$ of $\lambda_\star (G)$.   \end{theorem}

If a $G$-Galois extension has a Hopf Galois structure of type $N$, we'll say that the ordered pair $(G, N)$ of abstract groups (of equal order) is realizable. 

There has been much interest since the mid '90s around the qualitative question of realizability, and if so,  the quantitative question of counting the number of Hopf Galois structures of type $N$ on a $G$-Galois extension.

\subsection*{Skew left braces}

\begin{definition}  A  skew left  brace (or for short, skew brace) is a finite set $B$ with two operations, $\star$ and $\circ$, so that $(B, \star)$ is a group (the ``additive group''), $(B, \circ)$ is a group, and the compatibility condition
\[ a \circ (b \star c) = (a \circ b) \star a^{-1} \star (a \circ c) \]
holds for all $a, b, c$ in $B$.  Here $a^{-1}$ is the inverse of $a$ in $(B, \star)$. Denote the inverse of $a$ in $(B, \circ)$ by $\overline{a}$.  \end{definition}
If $B$ has two operations $\star$ and $\circ$ and is a skew brace with $(B, \star)$ the additive group, then we write $B = B(\circ, \star)$  (i. e. the additive group operation is on the right).

Skew braces were introduced in [GV17] as a non-commutative generalization of  left braces of [Rum07], which in turn generalize radical algebras.  Skew braces have been studied in [Bac16a] and in [SV18].  Initial interest in braces and skew braces was motivated by the search for set-theoretic solutions of the Yang-Baxter equation.  But there is also a close connection between skew braces and Hopf Galois structures on Galois extensions of fields.  This connection evolved from the discovery by [CDVS06] of the relationship between radical algebras and regular subgroups of the affine group, and its subsequent generalization and application to abelian Hopf Galois structures  on elementary abelian Galois extensions of fields in [FCC12].   Bachiller in [Bac16] observed that the connection in [FCC12] extends to a close relationship between abelian Hopf Galois structures on Galois extensions of fields and left braces.  That relationship  was extended to skew braces and arbitrary Hopf Galois structures in the appendix by Byott and Vendramin in [SV18], and has already been used to study Hopf Galois structures, for example in [Ze18] and [Ch18]. 

Associated to a set $B$ with two group operations $ \circ$ and $\star$  are the two left regular representation maps:
 \[ \begin{aligned} \lambda_\star:  &B\to \mathrm{Perm}(B), \lambda_\star(b)(x) = b \star x,\\
 \lambda_{\circ}: & B \to \mathrm{Perm}(B), \lambda_{\circ}(b)(x) = b \circ x. \end{aligned} \]
Let $\mathcal{L}:  B \to \mathrm{Perm}(B)$ by $\mathcal{L}_b(x) = \lambda_\star(b)^{-1}\lambda_{\circ}(b)(x)$. for $b, x$ in $B$.  Then  [GV17, Proposition 1.9] is
\begin{proposition}\label{3.1}  $(B, \circ, \star)$ is a skew brace if and only if  the group homomorphism  $\lambda_\circ:  (B , \circ) \to \mathrm{Perm}(B)$ has image in $\mathrm{Hol}(B, \star) = \lambda_\star(B) \mathrm{Aut}(B) \subset \mathrm{Perm}(B)$. \end{proposition}

\begin{proof}  For all $b, x, y$ in $B$, the left skew brace property:  
\[ b \circ (x \star y) = (b \circ x) \star b^{-1} \star (b \circ y) \]
holds   iff
\[ b^{-1} \star (b \circ (x \star y)) = (b^{-1}\star (b \circ x)) \star (b^{-1} \star (b \circ y)) ,\]
iff
\[ \lambda_\star^{-1}(b)\lambda_\circ(g)(x \star y)= (\lambda_\star^{-1}(b)\lambda_\circ(b)(x ))\star (\lambda_\star^{-1}(b)\lambda_\circ(b)( y))\]
iff 
\[ \mathcal{L}_b(x \star y) = \mathcal{L}_b(x)*\mathcal{L}_b(y) \]
iff $\mathcal{L}_b$ is in $\mathrm{Aut}(B, \star)$ for all $b$ in $B$, 

\noindent iff  $\lambda_\circ = \lambda_{\star}\mathcal{L}:  B \to \mathrm{Perm}(B)$ has image in $\lambda_\star(B)\mathrm{Aut}(B) = \mathrm{Hol}(B)$.
\end{proof}

For $(G, \circ, \star)$ a skew left brace, the map  
\[ \mathcal{L}:  (G, \circ) \to \mathrm{Aut}(G, \star)\]
defined by $g\mapsto \mathcal{L}_g$ is a group homomorphism:  see [GV17, Proposition 1.9, Corollary 1.10].   

In the brace literature the map $\mathcal{L}_g$ is often denoted by $\lambda_g$.   Here we reserve $\lambda$ for left regular representation maps. 

\section{Connecting skew braces with Hopf Galois structures}

Let $L/K$ be a Galois extension with Galois group $G = (G, \circ)$.  Hopf Galois structures on $L/K$ of a given type $(G, \star)$ correspond by Galois descent [GP87] to regular subgroups $N$ of $\mathrm{Perm}(G)$ isomorphic to $(G, \star)$ and normalized by $\lambda_\circ G$.

\begin{proposition}  Let $L/K$ be a Galois extension with group $G = (G, \circ)$.  Let $H$ be a $K$-Hopf algebra giving a Hopf Galois structure of type $M$ on $L/K$.  Then $(G, \circ)$ has a skew left brace structure with additive group $(G, \star) \cong M$. \end{proposition}

This follows  from Theorem \ref{111} by the characterization of skew left braces in Proposition \ref{3.1}.

Conversely, 

\begin{proposition}  Let $(G, \circ, \star)$ be a skew brace.  Let $L/K$ be a Galois extension with Galois group $(G, \circ)$.   Then $L/K$ has a Hopf Galois structure of type $(G, \star)$. \end{proposition}

\begin{proof}  Given the skew brace structure $(G, \circ, \star)$ on the Galois group $(G, \circ)$ of $L/K$, we have by Proposition \ref{3.1} that $\lambda_\circ(G)$ is contained in $\mathrm{Hol}(G, \star)$, and so the subgroup $N = \lambda_\star(G) \subset \mathrm{Perm}(G)$ is normalized by $\lambda_\circ (G)$.  Thus $N$ corresponds by Galois descent to a Hopf Galois structure on $L/K$ of type $(G, \star)$.   \end{proof}

But the correspondence between regular subgroups $N$ of $\mathrm{Perm}(\Gamma)$ isomorphic to $(G, \star)$ and isomorphism types of skew braces $(G, \circ, \star)$ with $(G, \circ) \cong \Gamma$ and $(G, \star) \cong N$ is not bijective.  We have (c. f. [Ze18], Corollary 2.4):

\begin{proposition} [Byott, Zenouz] \label{BZ} Given an isomorphism type $(B, \circ, \star)$ of skew left brace, the number of Hopf Galois structures on a Galois extension $L/K$ with Galois group isomorphic to $(B, \circ)$ and skew brace isomorphic to $(B, \circ, \star)$ is 
\[ \mathrm{Aut}(B, \circ)/\mathrm{Aut}_{sb}(B, \circ, \star).\]
\end{proposition}

Here $\mathrm{Aut}_{sb}(B, \circ, \star)$ is the group of skew brace automorphisms of $(B, \circ, \star)$, that is, maps from $B$ to $B$ that are simultaneously group automorphisms of $(B, \star)$ and of $(B, \circ)$.

\

  As noted earlier, there has been considerable interest among researchers in Hopf Galois theory on the questions, given a $G$-Galois extension of fields, what are the possible types of Hopf Galois structures on the extension, or, given a group $N$, for which $G$ is there a $G$-Galois extension of fields with a Hopf Galois structure of type $N$?  Concisely, which pairs $(G, N)$ are realizable?   There has also been  considerable interest among researchers on  braces and, more recently, skew braces, on the question, which pairs $(G,N)$  of finite groups can be the additive, resp. circle group of a skew brace?   The two question are the same.  As an indication of interest in the brace question,  we note that at least a dozen of the 50 problems on skew braces in Vendramin's recent manuscript, ``Problems on skew left braces" [Ven18] relate to that question.  

\section{Bi-skew braces}\label{s-rsb}

 A bi-skew brace  is a finite set $B$ with two operations in which either operation can define the additive group.  More precisely,

\begin{definition}  A bi-skew brace  is a finite set $B$ with two operations, $\star$ and $\circ$ so that $(B, \star)$ is a group , $(B, \circ)$ is a group, and the two compatibility conditions
\[ a \circ (b \star c) = (a \circ b) \star a^{-1} \star (a \circ c) \]
and
\[ a \star (b \circ c) = (a \star b) \circ \overline{a} \circ  (a \star c) \]
hold for all $a, b, c$ in $B$.   \end{definition}
Thus $B$ with the two  operations is a skew left brace with either group acting as the additive group.                               

\begin{example}
A trivial example:  A group $G$ with operation $\star$ is a bi-skew brace with $\circ = \star$. 

Almost as trivial:  A group $G$ with operation $\star$ is a bi-skew brace with $\circ$ defined by $g \circ h = h * g$.  (See [SV18], Example 1.5.)  
\end{example}

We'll see many more examples below.

Given a skew brace $B$ with circle group isomorphic to $\Gamma$ and additive group isomorphic to $G$ the pair $(\Gamma, G)$ is realizable.  If $B$ is a bi-skew brace, then the pair $(G,\Gamma)$ is also realizable.

More generally we have the following quantitative result, an immediate consequence of Proposition \ref{BZ}.

\begin{proposition}  Let $(B, \circ, \star)$ be a bi-skew brace with $(B,  \circ) \cong \Gamma$ and $(B, \star) \cong G$.  Let $e_B(\Gamma, [G])$ be the set of Hopf Galois structures of type $G$ on a Galois extension $L/K$ with Galois group $\Gamma$ where the Hopf Galois structures come from $B$,and let $e_B(G, [\Gamma])$ be the set of Hopf Galois structures of type $\Gamma$ on a Galois extension $L'/K'$ with Galois group $G$ where the Hopf Galois structures come from $B$.  Then
\[ e_B(\Gamma, [G]) \cdot |\mathrm{Aut}(G)| = e_B(G, [\Gamma]) \cdot |\mathrm{Aut}(\Gamma)|. \]
\end{proposition}

\begin{proof}  $B$ be a skew brace with additive group $\cong G$ and circle group $\cong \Gamma$.  Then by Proposition \ref{BZ},  
\[  e_B(\Gamma, [G]) = \frac {|\mathrm{Aut}(\Gamma)|}{|\mathrm{Aut}_{sb}(B)|}.\]
For a  bi-skew brace $B$,  the group of skew brace automorphisms of $(B, \circ, \star)$ is identical to the group of skew brace automorphisms of $(B, \star,\circ)$. The result follows immediately. \end{proof}

\section{Bi-skew braces  from radical algebras}
Here is a  collection of non-trivial examples of bi-skew braces.

\begin{proposition}  Let $(A, +, \cdot)$ be a nilpotent $\mathbb{F}_p$-algebra of $\mathbb{F}_p$-dimension $n$.  Define the circle operation on $A$ by 
\[ a \circ b = a + b + a\cdot b.\]
Then $(A, \circ, +)$ is a bi-skew brace if and only if $A^3 = 0$ (i. e. for every $a, b, c$ in $A$, $a\cdot b \cdot c = 0$).
\end{proposition}

\begin{proof} 
As observed in [Ru07],  every nilpotent $\mathbb{F}_p$-algebra yields a left brace $(A, \circ, +)$, because for all $a, b, c$ in $A$,
\[ a \circ (b + c)  = a + b + c + ab + ac = (a \circ b) - a + (a \circ c).\]
To consider the possibility that $(A,  +, \circ,)$ is also a skew brace we  first observe that in $A$, 
\[ a \circ b \circ c = (a + b + c) + (a\cdot b + a \cdot c + b \cdot c) + a \cdot b \cdot c.\]
Now consider the  skew brace condition (*):
\[a + (b \circ c) = (a + b) \circ \overline{a} \circ (a + c) ,\]
where $\overline{a}$ is the $\circ$-inverse of $a$.  If this holds for all $a, b, c$, then it holds modulo the ideal $A^4$.  

The left side of (*) is 
\[ a + b + c + b\cdot c .\]
The right side of (*) is
\[ \begin{aligned} &(a +b) + \overline{a}  + (a +c)\\
& + (a + b)\cdot \overline{a} + (a + b) \cdot (a + c) +  \overline{a} \cdot (a + c)\\
& +  (a + b) \cdot \overline{a} \cdot (a + c).\end{aligned} \]
Modulo $A^4$, $\overline{a} = -a + a^2 - a^3$ (where $a^n = a \cdot a \cdot \ldots \cdot a$ ($n$ factors).  Viewing the right side modulo $A^4$, we obtain
\[ \begin{aligned} &(a +b) + (-a + a^2 - a^3)  + (a +c)\\
& + (a + b)\cdot (-a + a^2) + (a + b) \cdot (a + c) +  (-a + a^2) \cdot (a + c)\\
& +  (a + b) \cdot (-a) \cdot (a + c).\end{aligned} \]
In the algebra $(A, +, \cdot)$ this last expression reduces to 
\[ a + b + c + b \cdot c + b \cdot a \cdot c .\]
Thus, if $A^3 = 0$, then $b \cdot a \cdot c = 0$ and  (*) holds for all $a, b, c$ in $A$.  On the other hand, if $A^3 \ne 0$, then there exist $a, b, c$ so that $b \cdot a  \cdot c \ne 0$ in $A$, hence in $A/A^4$.  So the skew brace condition (*) fails in $A$.
\end{proof} 

To get a sense of how many  examples there are, we note:

\begin{theorem} (Kruse-Price)[KP70, Theorem 2.2]  The number of isomorphism classes of $\mathbb{F}_p$-algebras $A$ of dimension $n$ with $A^3 = 0$ is $p^\alpha$ where $\alpha = \frac 4{27}n^3 + O(n^2)$. \end{theorem}

In [DeG17], de Graaf determines all isomorphism types of radical algebras $A$ of dimension $\le 4$ over a finite field.  Some of the algebras $A$ are non-commutative and have $A^3 = 0$.  
Here is a dimension 3 example.  

\begin{example}  Let $ A = A_{3, 4}^0$ be the three-dimensional $\mathbb{F}_p$-algebra of de Graaf.  Then 
\[ A_{3, 4}^0 = \langle a, b, c | a^2 = c, ab = c\rangle \]
(so all other products  of two of $a, b, c$ are zero).  Note that $A^3 = 0$.

Let $x = c, y = b, z = a-b$, then  the multiplication table in $A$ on those $\mathbb{F}_p$-basis elements is given by 
\begin{center}\begin{tabular}
{c | c |c |c} \\
$\cdot$ &$x$ & $y$ & $z$ \\ \hline
$x$ & 0 & 0 & 0\\\hline
$y$  & 0 & 0 & 0\\\hline
$z$ & 0 & $x$ & 0 \\ \hline
\end{tabular}\end{center}
Mapping $A$ to $\mathbb{F}_p^3$ by sending $x, y, z$ to  the standard basis vectors in $\mathbb{F}_p^3$
and defining $\overline{u} \circ \overline{v} = \overline{u} + \overline{v} + \overline{u} \cdot \overline{v}$ in $A$ as usual, we find that 
\[\begin{pmatrix} x_1\\y_1\\z_1 \end{pmatrix} \circ \begin{pmatrix} x_2\\y_2\\z_2 \end{pmatrix} = \begin{pmatrix} x_1 + x_2 + z_1y_2\\y_1+ y_2\\z_1 + z_2 \end{pmatrix}.\]
This is the circle operation for the second brace with additive group $(\mathbb{Z}/(p))^3$ and socle of order $p^2$ of Theorem 3.2 of Bachiller's classification of left braces in [Ba14]. The circle group $(A, \circ)$ is isomorphic to the Heisenberg group, the unique non-abelian group of order $p^3$ and exponent $p$.

Since  $A^3 = 0$, this brace is a bi-skew brace:  $(A, +, \circ)$ is a skew brace with  additive group $(A, \circ)$ isomorphic to the Heisenberg group.
\end{example}

\section{Bachiller's classification of left braces of order $p^3$}

A left brace $(B, +, \circ)$ is two-sided if
\[ ((a + b)\circ c)  = (a \circ c) - c + (b \circ c)\]
for all $a, b, c$ in $B$.  If we define a multiplication $x \cdot y$ by 
\[ x \circ y = x + y + x \cdot y, \]
then the formula above is equivalent to right distributivity in $(B, +, \cdot)$.  Rump [Rum07] showed that a brace is two-sided if and only if $(B, +, \cdot)$ is a radical ring.

Bachiller [Ba14] classified the left braces $B = (B, \circ, +)$ of order $p^3$. In this section we show that the left braces classified by Bachiller that are bi-skew braces are exactly the braces that are two-sided, hence that are radical rings.

The additive group $(B, +)$ of a left brace of order $p^3$ must be isomorphic to $\mathbb{Z}/(p^3), \mathbb{Z}/(p) \times \mathbb{Z}/(p^2)$ or $(\mathbb{Z}/(p))^3$.  We divide the examination into three cases.

Case 1.     $(B, +) = (\mathbb{Z}/(p))^3$.  For $p > 3$ there are eight forms of isomorphism types of left braces $(B, +, \cdot)$, some with parameters.   For such a brace, write
\[ (B, +) = \{\mathbf{u} = \begin{pmatrix} x\\y\\z \end{pmatrix} :  x, y, z \in \mathbb{Z}_(p)\}.\]
Then  the circle operation in every case has the form 
\[ \mathbf{u}_1 \circ \mathbf{u}_2 = \mathbf{u}_1 + \mathbf{u}_2 + \begin{pmatrix} f\\g\\h \end{pmatrix} \]
where 
\[ \begin{aligned} f &= \mathbf{u}_1^T M_f \mathbf{u}_2 + q_f(\mathbf{u}_1, \mathbf{u}_2)\\
g &= \mathbf{u}_1^T M_g \mathbf{u}_2 + q_g(\mathbf{u}_1, \mathbf{u}_2)\\
h &= \mathbf{u}_1^T M_h \mathbf{u}_2 + q_h(\mathbf{u}_1, \mathbf{u}_2)\end{aligned} \]
where $M_f, M_g, M_h$ are in $M_3(\mathbb{Z}/(p^3)$ and $q_f, q_g, q_h$ are quadratic functions of the components of $\mathbf{u}_1$.

Case 2.  $(B, +) = \mathbb{Z}/(p) \times \mathbb{Z}/(p^2)$.  For $p > 3$ there are fifteen forms of isomorphism types of left braces $(B, +, \cdot)$, some involving integer parameters.   For such a brace, write
\[ (B, +) = \{\mathbf{u} = \begin{pmatrix} x\\y \end{pmatrix} :  x \in \mathbb{Z}_(p), y \in \mathbb{Z}/(p^2)\}.\]
Then  the circle operation in every case has the form 
\[ \begin{aligned} \mathbf{u}_1 \circ \mathbf{u}_2 &=\begin{pmatrix} x_1\\y_1 \end{pmatrix} \circ \begin{pmatrix} x_2\\y_2 \end{pmatrix} \\&= \begin{pmatrix} x_1\\y_1 \end{pmatrix} + \begin{pmatrix} x_2\\y_2 \end{pmatrix} + \begin{pmatrix} f(\mathbf{u}_1, \mathbf{u}_2)\\pg(\mathbf{u}_1, \mathbf{u}_2)\end{pmatrix}\end{aligned} \]
where both $f(\mathbf{u}_1, \mathbf{u}_2)$ and $g(\mathbf{u}_1, \mathbf{u}_2)$ have the form
\[ (x_1, y_1)M\begin{pmatrix} x_2\\y_2 \end{pmatrix} + q(\mathbf{u}_1, \mathbf{u}_2)\]
where $M$ is in $M_2(\mathbb{Z}/(p))$ and $q(\mathbf{u}_1, \mathbf{u}_2)$ is a quadratic function of the components of $\mathbf{u}_1$.

Case 3.  $(B, +) = \mathbb{Z}/(p^3) = \{\mathbf{u} = (x) | x \in \mathbb{Z}/(p^3)\}$.  Then $B$ is a radical algebra with multiplication $(x_1)\cdot (x_2) = (p^rx_1x_2)$, and 
\[  \mathbf{u_1} \circ \mathbf{u_2} = (x_1)\circ (x_2) = (x_1 + x+2  + p^rx_1x_2) \]
for $r = 0, 1, 2$.  If $r = 1$, then 
\[ (x_1)\circ (x_2) \circ (x_3) = (x_1 + x_2 + x_3 + p^2x_1x_2x_3),\]
so $B^3 \ne 0$.  In the cases $r = 0, 2$,  $B^3 = 0$, so $(B, \circ, +)$ is also a skew brace.

We have
\begin{proposition}  In Cases 1 and 2, a left brace  $B = (B, \circ, +)$ of order $p^3$ as classified in [Ba14] is a radical algebra if and only if the quadratic functions $q_f, q_g$, resp. $q_f, q_g, q_h$ are zero.  \end{proposition}

\begin{proof} To see if a given left brace in Bachiller's classification is a radical ring, it suffices to check right distributivity, where the multiplication $\cdot$ in $(B, \circ, +)$ is defined by 
\[ \mathbf{u}_1 \circ \mathbf{u}_2 = \mathbf{u}_1 + \mathbf{u}_2 + \mathbf{u}_1 \cdot \mathbf{u}_2.\]
One sees routinely that for each case, right distributivity holds if and only if the functions that are quadratic in the components of $\mathbf{u}_1$ are zero.
\end{proof}

One can  see easily that each  radical ring $(B, +, \cdot)$ in Cases 1 and 2 has $(B, \cdot)^3 = 0$.  Hence,

\begin{corollary}  With one exception, every left brace $(B, \circ, +)$ of order $p^3$ where $(B, +,\cdot)$ is a radical ring yields a bi-skew brace. \end{corollary}

The exception is the brace with additive group $\mathbb{Z}/(p^3)$ and circle operation $x \circ y = pxy$.  

Of the 26 isomorphism forms of left braces of order $p^3$ found by Bachiller,  two forms have $(B, +) \cong \mathbb{Z}_(p^3)$ and are bi-skew braces with $(B, \circ)\cong \mathbb{Z}_(p^3)$, seven forms have $(B, +) \cong \mathbb{Z}_{(p)} \times \mathbb{Z}_{(p^2)}$, $(B, \circ) \cong M_3(p)$ and are bi-skew braces, and four forms have $(B, +) \cong \mathbb{Z}_{p}^3 $, $(B, \circ) \cong M_{(p)}$ and are bi-skew braces.  Here $M_3(p)$ is the unique non-abelian group of order $p^3$ and exponent $p^2$, while $M_{(p)}$ is the Heisenberg group, of order $p^3$ and exponent $p$.

\begin{remark}  For each of the four bi-skew braces $A = (A, \circ, +)$ arising from a non-commutative radical algebra $A$ with $A^3 = 0$,  $(A, +) = C_p^3$ and $(A, \circ) = M_{(p)}$, the Heisenberg group of order $p^3$, Proposition \ref{BZ} yields the quantitative formula
\[e_A(C_p^3, [M_{(p)}])\cdot \mathrm{Aut}(M_{(p)})= e_A (M_{(p)}, [C_p^3])\cdot \mathrm{Aut}(C_p^3) .\]
Now $|Aut(C_p^3| = (p^3 -1)(p^2-1)(p-1)$ and $|Aut(M_{(p)})| = |C_p^3||GL_2(\mathbb{F}_p)| = p^3\cdot (p^2 -1)(p-1)$,  so
\[ e_A (M_{(p)}, [C_p^3]) =(p^3-1)e_A (M_{(p)}, [C_p^3]).\]
In words, given Galois extensions $K/k$ with Galois group $\Gamma \cong C_p^3$ and $K'/k'$ with Galois group $G \cong M_{(p)}$, for each Hopf Galois structure of type $C_p^3$ arising from the bi-skew brace $A$ on $K'/k'$, there are $p^3-1$ Hopf Galois structures of type $M_{(p)}$ arising from $A$ on $K/k$.   

The  numbers $e_B (C_p^3, [M_{(p)}])$ for skew braces $B$ with additive group $M_{(p)}$ and circle group $C_p^3$ may be found in Section 4 of [Ze18].
\end{remark}

\section{On Bachiller's exponent result}

Let $p$ be prime and $B = (B,\circ, +)$ be a brace of order $p^n$.  Then $(B, +)$ is an abelian $p$-group, hence of the form 
\[ (B, +) = \mathbb{Z}_{p^{a_1}} \times \mathbb{Z}_{p^{a_2}}\times \ldots \times \mathbb{Z}_{p^{a_m}} \]
with $a_1 \le a_2 \le \ldots \le a_m$.   Generalizing  [FCC12] for $(B, +)$ a radical ring, Bachiller [Bac16] proved that if $m + 2 \le p$, then for every $b$ in $B$, the order of $b$ in $(B, \circ)$ is equal to the order of $b$ in $(B, +)$.   In particular, if $\circ$ is commutative, then $(B, \circ) \cong (B, +)$.  This result is illustrated with the examples in the last section. 

The limitation on $m$ and $p$ is necessary:  see [Ch07], which constructs examples equivalent to left braces $(B, \circ, \star)$ where $(B, \star) \cong \mathbb{Z}_p^m$ and $(B, \circ)$ is the group of principal units of $\mathbb{F}_p[x]/(x^{m+1})$ and has exponent $p^e$ where $p^{e-1} \le m < p^e$.  

Given  an $\mathbb{F}_p$-algebra $A$ with  $(A, +)$ of order $p^n$, $(A, +)$ has exponent $p$, so if $n+ 2 \le p$, then  $(A, \circ)$ has exponent $p$.  Thus for large $p$, the possible isomorphism types of Galois groups $G$ of order $p^n$ that have Hopf Galois structures arising from identifying $G$ as the circle group of an $\mathbb{F}_p$-algebra are among those groups of order $p^n$ and exponent $p$.   There is some fairly recent literature counting the  groups of order $p^n$ for $p > n+ 2$ and  $n \le 7$.  This literature is reviewed in [VL14], a paper in which Vaughan-Lee gives an explicit polynomial of degree 4 in $p$, 
\[ p^4 + 2p^3 + 147p^2 + (3p + 29)gcd(p-1, 3) + 5gcd(p-1, 4) + 1246,\]
 that counts the number of groups of order $p^8$ and exponent $p$.   How many of these groups can be the circle group of an $\mathbb{F}_p$-algebra of dimension 8 is unknown (to me).  On the other hand, Bachiller [Bac16] has found a group of order $p^{10}$ and exponent $p$ for $p > 12$ that is not the circle group of a brace with additive group an elementary abelian group of order $p^{10}$.

\section{Bi-skew braces and semidirect products}
Now we consider another large class of bi-skew braces.

Let $G$ be a finite group with a pair of complementary subgroups $G_L$, $G_R$.  This means:   $G_L \cap G_R = {1}$ and $|G_L||G_R| = |G|$, or equivalently, each $g$ in $G$ is uniquely a product $g = g_L\cdot g_R$ with $g_L$ in $G_L$, $g_R$ in $G_R$.  Byott [By15] observed that for any such group $G = G_L\cdot G_R$, there is a Hopf Galois structure of type $G$ on a Galois extension with Galois group $G_L \times G_R$.  For let $\beta_L, \beta_R$ be the pair of homomorphisms from $G_L \times G_R$ to $G$ given by 
\[ \beta_L(g_L, g_R) = g_L, \  \  \beta_R(g_L, g_R) = g_R^{-1}.\]
Then $\beta:  G_L \times G_R \to \mathrm{Hol}(G)$ by 
\[ \beta(g_L, g_R) = \lambda(g_L)\rho(g_R) = \lambda(g_Lg_R^{-1})C(g_R)\]
is a regular embedding from $G_L \times G_R$ to $\lambda(G) \rtimes \mathrm{InAut}(G) \subset \mathrm{Hol}(G)$.  (Here $C(g)(x) = gxg^{-1}$ is conjugation by $g$.)  
The corresponding $\circ$-structure on $G(\cdot)$ is induced by $b:   G_L \times G_R$ given by $b(g_L, g_R) = \beta((g_L, g_R))(1) = g_Lg_R^{-1}$.
Thus
\[ \begin{aligned}  b(g_L, g_R)\circ b(h_L, h_R) 
&=  (g_Lg_R^{-1})\circ (h_Lh_R^{-1})\\
&= g_Lh_Lh_R^{-1}g_R^{-1}\\
& =(g_Lh_L)(g_Rh_R)^{-1}\\
& =b((g_Lh_L, g_Rh_R)\\
& =b((g_L, g_R)(h_L, h_R)).\end{aligned} \]
Then $(G, \circ, \cdot)$ is a left skew brace, where $(G, \cdot) = G = G_LG_R$ with its given operation, and $(G, \circ) \cong G_L \times G_R$.

It is easy to check that the skew brace relation
\[ x \circ (y \cdot z) = (x \circ y) \cdot x^{-1} \cdot (x \circ z) \]
holds:  both sides are equal to $xy_Lzy_R^{-1}$.  

Thus Byott's observation implies the existence of skew braces with additive group any group $G$ with complementary subgroups $G_L$ and $G_R$:  the circle group is isomorphic to $G_L \times G_R$.

Now consider the special case where  $G = G_L \rtimes G_R$ is a semidirect product of two finite groups $G_L$ and $G_R$, where $G_L$ is normal in $G$.   

Denote the group operation in $G$ by $\cdot$, which we will often omit.  Thus for $x, y$ in $G$, $xy = x \cdot y$.   

In  the semidirect product, an element $x_R$ of $G_R$ acts on $y_L$ in $G_L$ by conjugation:
\[ x_R^{-1} y_L = (x_R^{-1}y_Lx_R)x_R^{-1}.\] 
 Then 
\[ \begin{aligned} x y &= x_L x_R^{-1} y_L y_R^{-1} \\
& = x_L (x_R^{-1} y_Lx_R)x_R^{-1} y_R^{-1} .\end{aligned} \]
 
As above, we also define  the direct product operation $\circ$ on $G$ by 
\[ \begin{aligned}  x \circ y &=  x_L x_R^{-1} \circ y_Ly_R^{-1}\\
&= x_L y_L y_R^{-1}x_R^{-1}\\
&=  x_L y x_R^{-1}.\end{aligned} \]

For semidirect products, we have

\begin{proposition}\label{9.1}  $(G, \circ, \cdot)$ is a bi-skew brace. \end{proposition}

\begin{proof}  We need to check the relation

\[ x \cdot (y \circ z) = (x \cdot y) \circ \overline{x} \circ (x \cdot z)\ \ \ \ \qquad(*)\]
where if $x = x_Lx_R^{-1}$, then $\overline{x} = x_L^{-1}x_R$ is the $\circ$-inverse of $x$.  

Note that 
\[ \begin{aligned}  xy &= x_L x_R^{-1}y_L y_R^{-1}\\
& = x_L(x_R^{-1}y_L x_R)x_R^{-1}  y_R^{-1} ,\end{aligned} \]
so $(xy)_L = x_L(x_R^{-1}y_L x_R)$ and $(xy)_R^{-1} = x_R^{-1}  y_R^{-1} $.

Now the  left side of (*) is 
\[  x \cdot (y \circ z) = x y_Lzy_R^{-1}. \]
The right side of (*) is
\[ \begin{aligned} ( x y ) \circ \overline{x} \circ(  x z) &=  x  y \circ x_L^{-1}x_R \circ xz\\
&= xy \circ x_L^{-1} xz x_R \\
&= xy \circ x_R^{-1}z x_R\\
& = (xy)_L  x_R^{-1}z x_R (xy)_R^{-1}\\
&= ( x_L(x_R^{-1}y_L x_R))x_R^{-1}z x_R (x_R^{-1}  y_R^{-1} )\\
& =  x y_L z y_R^{-1} .\end{aligned} \]
Hence formula (*) holds and $(G, \circ, \cdot) $ is also a skew brace, with additive group $(G, \circ)$.

Thus given a semidirect product $G = G_L \rtimes G_R$, if we let $\cdot$ be the usual multiplication in $G$, and let $\circ$ be the multiplication on $G$ is given by $x_L x_R^{-1} \circ y_L y_R^{-1} = x_Ly_Ly_R^{-1}x_R^{-1} = x_Ly_L (x_Ry_R)^{-1}$, then $(G, \cdot, \circ)$ is a bi-skew brace. \end{proof}

This immediately yields a result of Crespo, Rio and Vela [CRV16]:

\begin{corollary}\label{CRV}  Let $G = H \rtimes J$ be a semidirect product and let $\Gamma = H \times J$.  Then $(H \rtimes J, H \times J)$ is realizable:   every $G$-Galois extension of fields has a Hopf Galois structure of type $\Gamma$. \end{corollary}

[CRV16] obtained this result by their method of induced Hopf Galois structures.

We note some  examples related to previous work on Hopf Galois structures.

\begin{example}  
 In [AB18], Alabdali and Byott look at Hopf Galois structures on a $\Gamma$- Galois 
extension $K/k$ of squarefree degree $n$.  They show that if $\Gamma = C_n$  is cyclic, then for every group $G$ of order $n$, $(C_n, G)$ is realizable, and in fact they count the number of Hopf Galois structures of type $G$ on $K/k$.    

They observe that every group $G$ of squarefree order $n$ is metabelian, that is, there is an abelian normal subgroup $A$ of $G$ so that $G/A$ is abelian.  Since $n$ is squarefree, necessarily $r = |A|$ and $s = |G/A|$ are coprime.  By the Schur-Zassenhaus Theorem, it follows that $G$ is a semidirect product, $G = A \rtimes G/A$.  Thus, by the discussion above on groups with complementary subgroups,  for every group $G$ of squarefree order $n$, $(C_n, G)$ is realizable.  

It follows from Corollary \ref{CRV} that $(G, C_n)$ is also realizable, that is,

 \begin{corollary}  Every Galois extension $K/k$ of squarefree order $n$ has a Hopf Galois structure of cyclic type.\end{corollary}
 \end{example} 
\begin{example}    In [BC12],  we let $L/K$ be a  Galois extension with Galois group $\Gamma$, a non-cyclic abelian $p$-group of order $p^n$, $n \ge 3$.  We showed that $L/K$ admits a nonabelian Hopf Galois structure.  The idea is to write $\Gamma = A \times B$, in such a way that there exists a non-abelian semidirect product $G = A \rtimes B$.  Thus $(\Gamma, G)$ is realizable.    It follows from Corollary \ref{CRV}  that $(G, \Gamma)$ is also realizable. More generally, 

\begin{corollary}  Given $G = A \rtimes B$ where $A$, $B$ are abelian, then $(G, A \times B)$ is realizable.  Thus a $G$-Galois extension admits an abelian Hopf Galois structure.  \end{corollary} 
\end{example} 

\begin{example}   Moving away from the abelian case, an obvious class of examples is to take $G = \mathrm{Hol}(N) \cong N \rtimes \mathrm{Aut}(N)$  for any finite group $N$.  Letting $\Gamma = N \times \mathrm{Aut}(N)$, we have that both $(\Gamma, G)$ and $(G, \Gamma)$ are realizable.  If $N$ is non-abelian, then both $\Gamma$ and $G$ are non-abelian.

For a small example, let $G = M_{(p)} \rtimes  \mathrm{Aut}(M_{(p)})$, and let $\Gamma = M_{(p)} \times \mathrm{Aut}(M_{(p)})$.  
Here $M_{(p)}$ is the Heisenberg group over $\mathbb{F}_p$, which can be identified as the subgroup of $\mathrm{GL}_3(\mathbb{F}_p)$ consisting of upper triangular $3 \times 3$ matrices with diagonal entries $= 1$, and $\mathrm{Aut}(M_{(p)})$ is isomorphic to $C_p^2 \rtimes \mathrm{GL}_2(\mathbb{F}_p)$ [Ze18].  Then $G$ is a bi-skew brace of order $p^6(p^2-1)(p-1)$ with both additive and circle groups non-abelian. 
\end{example} 
\begin{remark}  
We note that Proposition \ref{9.1} need not hold if $G = HJ$ and $H$ is not a normal subgroup of $G$.  For a small example,   let $G = S_4 = S_3 \cdot C_4$ where $S_3 = \mathrm{Perm}({1, 2, 3})$ and $C_4$ is generated by the cycle $(1, 2, 3, 4)$.   Letting $\star$ be the usual operation in $S_4$ and $\circ$ be the operation in $S_4$ induced from the componentwise operation on $S_3 \times C_4$,  $(S_4, \circ, \star)$ is a skew brace.   But, as is easily checked, the skew brace defining relation for $(S_4, \star, \circ)$ does not always hold.   For example, letting $x = x_R^{-1} = (1 2 3 4)$, $y = y_L = (1 2)$ and $z = z_R^{-1} = (13)(24)$  (cycle notation),  I found that $x\star (y \circ z) = (132)$, while 
\[ x \star y \circ \overline{x} \circ x \star z = (134).\]
\end{remark}

\begin{remark} We observe that there are bi-skew braces arising from radical algebras that do not arise from semidirect products.   

 A bi-skew brace arising from a semidirect product $G = H \rtimes J$ of groups as above has $(G, \star)$  a semidirect product and $(G, \circ)$ a direct product $H \times J$. Thus a bi-skew brace $(G, \circ, \star)$ arising from a semidirect product $H \rtimes J$ has the property that $H$ and $J$ are subgroups of both $(G, \circ)$ and $(G, \star)$, where the operations $\circ$ and $\star$ coincide on  $H$ and, if $J$ is abelian,  on $J$.  

A skew brace $(A, \circ, +)$ arising from a radical $\mathbb{F}_p$-algebra $(A, +, \cdot)$  has additive group $(A, +)$ isomorphic to $C_p^n$. Thus if the skew brace $(A, \circ, +)$ arising from a radical $\mathbb{F}_p$-algebra  also arose from a semidirect product, then the group $(A, +)$ would be an elementary abelian group, and the group $(A, \circ)$ would be a semidirect product of two elementary abelian subgroups, where on the subgroups, $\circ$ and $+$ are the same and the multiplication $\cdot$ is trivial.  

Consider the radical $\mathbb{F}_p$-algebra $A$ of dimension 6, with generators $x, y, z, a, b, c$  and relations
\[xy = a, yz =  b, zx = c\]
with all other products of generators $= 0$.  Then the center of $(A, \circ)$ is $\langle a, b, c \rangle = A^2$, on which $\circ$ is the same as $+$.

Let $\alpha = rx + sy + tz + d$ where $d$ is in $A^2$ and $r, s, t$ are in $\mathbb{F}_p$.    Then 
\[ \alpha \circ \alpha = \alpha + \alpha + rsa + stb + trc.\]
 So $\alpha + \alpha = \alpha \circ \alpha$ if and only if $\alpha\cdot \alpha = 0$, if and only if at least two of $r, s$ and $t$ are zero.  Also, since $x, y$ and  $z$ do not commute pairwise in $(A, \cdot)$, any subgroup of $(A, +)$ containing two of $x, y$ and $z$ will not have $\circ$ = $+$.

Thus if $H$ and $K$ are $\circ$-subgroups of $A$ on which $\circ = +$, then $H \times K$  can contain at most two of $x, y$ and $z$.  So $H \times K$ cannot equal $(A, +)$.
\end{remark}

\end{document}